\documentclass{article}

\usepackage{graphicx}      
\usepackage{natbib}        

\usepackage{amsthm}

\newtheorem{theorem}{Theorem}[section]
\newtheorem{corollary}{Corollary}
\newtheorem{lemma}[theorem]{Lemma}
\newtheorem{proposition}{Proposition}

\theoremstyle{definition}

\usepackage{xcolor} 

\usepackage{amssymb}
\usepackage{mathrsfs}
\usepackage[cmex10]{amsmath}

\newcommand{\memo}[1]{}

\newcommand{\hs}[1]{\hspace{#1mm}}
\newcommand{\nt}{\notag}

\newcommand{\pd}{\partial}
\newcommand{\pr}{\prime}

\newcommand{\mcal}[1]{{\mathcal{#1}}}
\newcommand{\mbb}[1]{{\mathbb{#1}}}

\newcommand{\vl}[1]{\mbox{\boldmath $#1$}}

\def\p{p}

\def\a{\lambda}

\def\topS{{\!\top\!}}

\def\org{0}
\def\CD{\!\phantom{\ }_\org D^\alpha_t}

\newcommand{\CDg}[2]{\!\phantom{\ }D^{#2}_{#1}}

\newcommand{\CDh}[3]{\!\phantom{\ }_{#1}^{\phantom{1}} D^{#3}_{\vphantom{1}#2}}

\def\ML{E}
\def\Re{\mathrm{Re}}

\def\fh{f}

\newcommand{\Sp}{\vl{\varsigma}}

\memo{
\makeatletter
\newenvironment{proof}[1][\proofname]{\par
  \normalfont
  \topsep0\p@\@plus6\p@ \trivlist
  \item[\hskip\labelsep{\itshape #1}\@addpunct{\itshape :}]\ignorespaces
}{%
  \endtrivlist
}
\newcommand{\proofname}{Proof}
\makeatother

} 

\title{Hamilton-Jacobi-Bellman equation of\\
 nonlinear optimal control problems\\ with fractional discount rate} 

\author{{Gou Nishida}\thanks{Department of Electrical and Electronic Engineering, College of Engineering, Nihon University, 1 Nakagawara, Tokusada, Tamura, Koriyama, Fukushima, 963-8642, JAPAN. (e-mail: g.nishida@ieee.org)}
\thanks{The authors acknowledge the support of JSPS KAKENHI Grant Number 22H01428.} 
\and
{Takamatsu Takahiro}\thanks{Fukushima Renewable Energy Institute, National Institute of Advanced Industrial Science and Technology (FREA), 2-2-9, Machiikedai, Koriyama, Fukushima 963-0298, JAPAN. (e-mail: takamatsu.1849@aist.go.jp)} 
\and
{Noboru Sakamoto}\thanks{Department of Mechatronics, Faculty of Science and Engineering, Nanzan University, 18 Yamazato-cho, Shyowa-ku, Nagoya, 466-8673, JAPAN. (e-mail: noboru.sakamoto@nanzan-u.ac.jp)} 
}

\date{11/Nov/2022}

\memo{
} 

\begin{document}
\maketitle

\begin{abstract}                
 This paper derives the Hamilton-Jacobi-Bellman equation of nonlinear optimal control problems for cost functions with fractional discount rate from the Bellman's principle of optimality. The fractional discount rate is described by Mittag-Leffler function that can be considered as a generalized exponential function.
\end{abstract}


\section{Introduction}

The Hamilton-Jacobi-Bellman (HJB) equation in optimal control theory yields
a necessary and sufficient condition for the optimality of controls with respect to cost functions~\citep{And:90}.
The optimal controls are obtained from the HJB equation as a minimizer of the Hamiltonian included in the equation.
The cost functions can be variously chosen for each control problem.
The discounted cost is defined by multiplying an exponential function to an integrand of a cost function.
The exponential function affects evaluations to increase or decrease weighting as time progresses.

In this paper, the Hamilton-Jacobi-Bellman equation of nonlinear optimal control problems 
for an extended discounted cost function using a fractional discount rate described by Mittag-Leffler function~\citep{Pod:99,Gor:20}
is derived from the Bellman's principle of optimality~\citep{Bel:21}.

\section{Preliminaries}

This section summarizes the basic definition and derivation of the infinite horizon optimal control problem with discounted cost.


\subsection{Optimal problem with discount rate}

Consider the infinite horizon optimal control problem for the system
\begin{align}
 \dot{x}(t) = \fh(x(t), u(t), t)
\end{align}
with the cost function
\begin{align}
 J = \int_{t}^{\infty} {e^{\a(\tau-t)}} L(x(\tau), u(\tau), \tau)\,d\tau \label{Cost}
\end{align}
for some time $t$ such that $t_0 \leq t \leq \infty$, where $\a \in \mbb{R}$, $x=x(t) \in \mbb{R}^n$ is the state, $u=u(t) \in \mbb{R}^m$ is the control, $f\colon M \to \mbb{R}^n$ is a smooth mapping, and $M$ is an $n$-dimensional manifold.
If $\lambda > 0$, then the cost is expanded as time progresses, and designed controls work 
in such a way as to suppress states within a short time.
If $\lambda < 0$, prospective costs do not so much take into account with distance from an initial time.

This problem can be formulated by dynamic programming as follows:
{
\begin{align}
 &V(x,t) = \inf_{u[t,\infty]} J \notag \\
 &= \inf_{u[t,\infty]} \left( \int_{t}^{t^\pr} {e^{\a(\tau-t)}} L(x(\tau), u(\tau), \tau)\,d\tau \right. \notag \\
 &\qquad \left. + \int_{t^\pr}^{\infty} {e^{\a(\tau-t)}} L(x(\tau), u(\tau), \tau)\,d\tau \right) \notag \\
 &= \inf_{u[t,t^\pr]} \left\{ \int_{t}^{t^\pr} {e^{\a(\tau-t)}} L(x(\tau), u(\tau), \tau)\,d\tau \right. \notag \\
 &\qquad \left. + \inf_{u[t^\pr,\infty]} \left( \int_{t^\pr}^{\infty} {e^{\a(\tau-t)}} L(x(\tau), u(\tau), \tau)\,d\tau \right) \right\} \notag \\
%
%
 &= \inf_{u[t,t^\pr]} \left( \int_{t}^{t^\pr} {e^{\a(\tau-t)}} L(x(\tau), u(\tau), \tau)\,d\tau \right. \notag \\
 &\qquad \left. +\, {e^{\a dt}} V(x(t^\pr),t^\pr) \right), \label{eq3}
\end{align}
where we have used $t^\pr = t + dt$ in the last equation.
For an infinitesimal time $dt > 0$, the above $V$ called a value function can be written as follows:
} 
\begin{align}
 &V(x,t) \nt \\ 
 &\ = \inf_{u[t,t+dt]} \left( L(x, u, t)\,dt + {e^{\a dt}} V(x+ \fh dt,t+dt) \right), \label{eq4}
\end{align}
where we have assumed that $V$ called a value function is smooth.
From the first order term of the Taylor series expansion of (\ref{eq4}):
\begin{align}
 &V(x+ \fh dt,t+dt) \notag \\
 &\quad = V(x,t) + \dfrac{\partial V}{\partial x}(x,t) \fh dt + \dfrac{\partial V}{\partial t}(x,t)dt + \mathcal{O}(dt^2), \\
 &{e^{\a dt} = 1 + \a dt + \mathcal{O}(dt^2)},
\end{align}
 we obtain
\begin{align}
 {-\a V(x,t)} - \dfrac{\partial V}{\partial t}(x,t) &= \inf_{u} \left( L(x, u, t) + \dfrac{\partial V}{\partial x}(x,t) \fh \right) \nt \\
 &= \inf_{u} H(x, u, \p, t), \label{HJ0}
\end{align}
where we have defined 
\begin{align}
 H(x, u, \p, t) = L(x, u, t)  + \dfrac{\partial V}{\partial x}(x,t) \fh \label{preH}
\end{align}
for $\p = (\partial V/ \partial x)^\top(x,t)$.
The equation (\ref{HJ0}) is the Hamilton-Jacobi-Bellman equation with discounted cost that is the necessary condition of the optimal problem.
The term $-\a V(x,t)$ does not appears in the standard optimal control problem setting.

\memo{
The sufficient condition, i.e., the fact that the optimal control and the value function are given if the Hamilton-Jacobi equation is solvable can be also proven.
Indeed, for an optimal control input $u_{\textrm{opt}}$ that minimize the right-side of (\ref{HJ0}) and any input $u$, the relation
\begin{align}
 H(x, u, \p, t) &\geq H(x, u_{\textrm{opt}}, \p, t) \nt \\
 &= {-\a V(x,t)} - \dfrac{\partial V}{\partial t}(x,t)
\end{align}
holds. By using (\ref{preH}), we get
\begin{align}
  L(x, u, t) &\geq {-\a V(x,t)} - \dfrac{\partial V}{\partial x}(x,t)\fh - \dfrac{\partial V}{\partial t}(x,t) \notag \\
 &= {-\a V(x,t)} - \dfrac{d V}{dt}(x,t).
\end{align}
Thus, 
\begin{align}
 J &= \int_{t_0}^{\infty} {e^{\a(t-t_0)}} L(x(t), u(t), t)\,dt \notag \\
 &\geq \int_{t_0}^{\infty} {e^{\a(t-t_0)}} \left( {-\a V(x,t)} - \dfrac{d V}{dt}(x,t) \right)\,dt \notag \\
 &= - \int_{t_0}^{\infty} \dfrac{d}{dt} \left({e^{\a(t-t_0)}} V(x,t)\right)\,dt \notag \\ 
 &= V(x(t_0),t_0),
\end{align}
which is the minimum of the cost function.
} 


\section{Main results}


\subsection{Motivation and summary of results}

In (\ref{eq3}), the relation $e^{\a(\tau-t^\pr+dt)} = e^{\a(\tau-t^\pr)}e^{\a dt}$ has been used.
This relation is called the semigroup property of exponential functions.
In this paper, we attempt to use the Mittag-Leffler function 
$\ML_\alpha (\a(\tau-t^\pr+dt)^\alpha)$ instead of the exponential function in the optimal control problem.
The Mittag-Leffler function can be considered as a generalized exponential function.
For instance, $\ML_0(z) = 1/(1-z)$ for $|z|<1$, $\ML_1(z) = e^z$, $\ML_2(z) = \cosh(\sqrt{z})$ for $z \in \mbb{C}$, 
\memo{
$\ML_2(-z^2) = \cos(z)$ for $z \in \mbb{C}$, 
$\ML_3(z) = (1/2)[e^{z^{1/3}} + 2 e^{-(1/2) z^{1/3}}\cos(\sqrt{3} z^{1/3}/2)]$ for $z \in \mbb{C}$, 
$\ML_4(z) = (1/2)[\cos(z^{1/4}) + \cosh(z^{1/4})]$ for $z \in \mbb{C}$, and 
$\ML_{1/2}(\pm z^{1/2}) = e^z \mathrm{erfc} (\mp z^{1/2})$ for $z \in \mbb{C}$, 
} 
where $\mbb{C}$ is the set of complex numbers, and $\mathrm{erfc}$ is the error function.



\memo{
\subsection{Mittag-Leffler functions}

On the other hand, the solution of the fractional ordinary differential equation
\begin{align}
 \CD x(t) = \lambda x(t) + u(t),
\end{align}
for $t > 0$ and $x(0) = x_0$, is given uniquely by
\begin{align}
 x(t) &= \ML_\alpha(\lambda t^\alpha) x_0 \nt \\
&\quad + \int_0^t (t-\tau)^{\alpha -1} \ML_{\alpha, \alpha} [\lambda (t-\tau)^\alpha] u(\tau) d\tau,
\end{align}
where $\lambda \in \mbb{C}$, $u$ is a complex function defined in $(0, \infty)$, $x: [0, \infty) \to \mbb{C}$ is a continuous function, 
$\CD$ is the Caputo fractional time derivative of order $\alpha$ with $0 < \alpha < 1$
\begin{align}
 \CD x(t) = \dfrac{1}{\Gamma(1-\alpha)} \int_0^t (t - \tau)^{-\alpha} \dfrac{dx}{d\tau}(\tau)\,d\tau, \label{eqC}
\end{align}
} 

The Mittag-Leffler function $\ML_\alpha$ is defined by
\begin{align}
 \ML_\alpha(z) = \sum_{n=0}^\infty \dfrac{z^n}{\Gamma(\alpha n + 1)}
\end{align}
for $z \in \mbb{C}$ and $\alpha \geq 0$, where 
$\Gamma$ is the Gamma function defined by 
\begin{align}
\Gamma(z) = \int_0^\infty t^{z-1} e^{-t}\,dt \quad (\Re\,z > 0)
\end{align}
for a complex number $z$ with positive real part ($\Gamma(n+1) = n!$ for a natural number $n$, and 
$\Gamma(r+1) = r\Gamma (r)$ for a positive real $r$).

On the other hand, there is a difficulty that the semigroup property of the Mittag-Leffler function does not hold, i.e., 
$\ML_\alpha [\a(\tau-t^\pr+dt)^\alpha] \neq \ML_\alpha [ \a (\tau -t^\pr)^\alpha ] \ML_\alpha (\a dt^\alpha)$.
Then, the following relation holds~\citep{Gus:16}:
\begin{align}
 &\ML_\alpha [\a(t+s)^\alpha] = \ML_\alpha (\a t^\alpha) \ML_\alpha (\a s^\alpha) - \varDelta \ML_\alpha (t,s), \label{semiML} \\
 &\varDelta \ML_\alpha (t,s) = \int_0^t (\tau)^{\alpha-1} \ML_{\alpha, \alpha}(\a \tau^\alpha) F(t-\tau)\,d\tau, \\
 &F(t) = \int_0^s \dfrac{(t+s-\sigma)^{-\alpha}}{\Gamma(1-\alpha)} \dfrac{d \ML_\alpha}{d \sigma} (\a \sigma^{\alpha})\,d\sigma,
\end{align}
 and ${d \ML_\alpha}(\a \sigma^{\alpha})/{d \sigma} = \a \sigma^{\alpha - 1} \ML_{\alpha, \alpha}(\a \sigma^{\alpha})$ for $\sigma >0$, 
where $\ML_{\alpha, \alpha}$ is the generalized Mittag-Leffler function (or Wiman's function) defined by
\begin{align}
 \ML_{\alpha, \beta}(z) = \sum_{n=0}^\infty \dfrac{z^n}{\Gamma(\alpha n + \beta)}
\end{align}
for $z \in \mbb{C}$ and $\alpha, \beta \in \mbb{C}$ such that $\Re (\alpha) >0$ and $\Re (\beta) >0$.


\memo{
\subsection{Asymptotic expansions of Mittag-Leffler functions}

{\color{blue}

\begin{proposition}

\begin{align}
 \ML_{\alpha,\beta}(z) &= \dfrac{1}{\alpha} z^{\frac{1-\beta}{\alpha}} \exp \left( z^{\frac{1}{\alpha}} \right) \nt \\
 &\quad- \sum^N_{r=1} \dfrac{1}{\Gamma(\beta - \alpha r)} \dfrac{1}{z^r} + \mcal{O}\left[ \dfrac{1}{z^{N+1}} \right]
\end{align}
for $0 < \alpha < 2$ as $|z| \to \infty$, and $|\arg z| \leq \mu$ for $(\pi \alpha)/2 < \mu \min[ \pi, \pi \alpha ]$.

\begin{align}
 \ML_{\alpha,\beta}(z) &= \dfrac{1}{\alpha} \sum_n z^{\frac{1}{n}} \exp \left[ 
\exp \left( \dfrac{2n\pi i}{\alpha} \right) z^{\frac{1}{\alpha}} \right]^{1-\beta} \nt \\
 &\quad- \sum^N_{r=1} \dfrac{1}{\Gamma(\beta - \alpha r)} \dfrac{1}{z^r} + \mcal{O}\left[ \dfrac{1}{z^{N+1}} \right]
\end{align}
for $\alpha \geq 2$, where the first sum is taken over all integers $n$ such that $|\arg(z)+2\pi n| \leq (\alpha \pi)/2$.

\end{proposition}

} 

\begin{align}
 \ML_{\alpha,1}(z) &= \dfrac{1}{\alpha} \exp \left( z^{\frac{1}{\alpha}} \right) \nt \\
 &\quad- \sum^N_{r=1} \dfrac{1}{\Gamma(1 - \alpha r)} \dfrac{1}{z^r} + \mcal{O}\left[ \dfrac{1}{z^{N+1}} \right]
\end{align}
} 



\subsection{Optimal control problem with generalized discounted cost}

In this paper, we assume that the error term $\varDelta \ML_\alpha (t,s)$ for a small $s$ is sufficiently small.

Let a real number $\alpha$ be the fractional order such that $0 < \alpha < 1$.
Consider the infinite horizon optimal control problem for the system
\begin{align}
 \dot{x}(t) = \fh(x(t), u(t), t) \label{Sys1}
\end{align}
with the cost function
\begin{align}
 &J = \int_{t}^{\infty} \bar{\ML}_\alpha[\a(\tau-t)^\alpha] L(x(\tau), u(\tau), \tau)\,d\tau \label{Cost} \\
 &\bar{\ML}_\alpha[\a(\tau-t)^\alpha] = \ML_\alpha [\a(\tau-t)^\alpha] + \varDelta \ML_\alpha (\tau,-t)
\end{align}
for some time $t$ such that $t_0 \leq t \leq \infty$, where 
$\bar{\ML}_\alpha$ is the approximated Mittag-Leffler function, 
$\a \in \mbb{R}$, $x=x(t) \in \mbb{R}^n$ is the state, $u=u(t) \in \mbb{R}^m$ is the control, $f\colon M \to \mbb{R}^n$ is a smooth mapping, and $M$ is an $n$-dimensional manifold.


\begin{proposition}
The infinite horizon optimal control problem for the system (\ref{Sys1}) with respect to the cost function (\ref{Cost})
can be formulated by 
\begin{align}
 V(x,t) &= \inf_{u[t,t+dt]} \left( L(x, u, t)\,dt \right. \nt \\
 &\qquad \qquad \left. + {\ML_\alpha (\a dt^\alpha)} V(x+\fh dt,t+dt) \right), \label{eq4a}
\end{align}
where we have assumed that $V$ is smooth, and the difference $\varDelta \ML_\alpha (t,s)$ between $\ML_\alpha [\a(t+s)^\alpha]$ and $\ML_\alpha (\a t^\alpha) \ML_\alpha (\a s^\alpha)$ in (\ref{semiML}) is sufficiently small.
\end{proposition}

\begin{proof}
Consider the following direct calculation by dynamic programming:
\begin{align}
 &V(x,t) = \inf_{u[t,\infty]} J \notag \\
 &= \inf_{u[t,\infty]} \left( \int_{t}^{t^\pr} {\bar{\ML}_\alpha[\a(\tau-t)^\alpha]} L(x(\tau), u(\tau), \tau)\,d\tau \right. \notag \\
 &\qquad \left. + \int_{t^\pr}^{\infty} {\bar{\ML}_\alpha[\a(\tau-t)^\alpha]} L(x(\tau), u(\tau), \tau)\,d\tau \right) \notag \\
 &= \inf_{u[t,t^\pr]} \left\{ \int_{t}^{t^\pr} {\bar{\ML}_\alpha[\a(\tau-t)^\alpha]} L(x(\tau), u(\tau), \tau)\,d\tau \right. \notag \\
 &\qquad \left. + \inf_{u[t^\pr,\infty]} \left( \int_{t^\pr}^{\infty} {\bar{\ML}_\alpha[\a(\tau-t)^\alpha]} L(x(\tau), u(\tau), \tau)\,d\tau \right) \right\} \notag \\
%
%
 &= \inf_{u[t,t^\pr]} \left( \int_{t}^{t^\pr} {\bar{\ML}_\alpha[\a(\tau-t)^\alpha]} L(x(\tau), u(\tau), \tau)\,d\tau \right. \notag \\
 &\qquad \left. +\, {\ML_\alpha (\a dt^\alpha)} V(x(t^\pr),t^\pr) \right), \label{eq3a}
\end{align}
where we have used $t^\pr = t + dt$ and 
\begin{align}
&\bar{\ML}_\alpha[\a(\tau-t)^\alpha]\mid_{t=t^\pr-dt} \nt \\
&\quad = \ML_\alpha [\a(\tau-t^\pr+dt)^\alpha] + \varDelta \ML_\alpha (\tau-t^\pr, dt) \nt \\
&\quad = \ML_\alpha [ \a (\tau -t^\pr)^\alpha ] \ML_\alpha (\a dt^\alpha)
\end{align}
in the last equation.
Then, for an infinitesimal time $dt > 0$, the above value function $V$ can be written as (\ref{eq4a}). 
\end{proof}


\begin{lemma}
The following relation holds:
\begin{align}
&{\ML_\alpha (\a dt^\alpha)} V(x + \fh dt,t + dt) \nt\\
&\quad \approx V + \dfrac{\partial V}{\partial x}\fh dt + \dfrac{\partial V}{\partial t} dt 
+ \dfrac{\a dt^{\alpha}}{\Gamma(\alpha + 1)}V  \nt \\
&\qquad + \dfrac{\a dt^{\alpha}}{\Gamma(\alpha + 1)} \dfrac{\partial V}{\partial x}\fh dt
+ \dfrac{\a dt^{\alpha}}{\Gamma(\alpha + 1)} \dfrac{\partial V}{\partial t} dt. 
\end{align}
\end{lemma}

\begin{proof}
From the first order terms of the Taylor series expansion of (\ref{eq4a}), 
\begin{align}
 &V(x+\fh dt,t+ dt) \notag \\
 &\ = V(x,t) + \dfrac{\partial V}{\partial x}(x,t)\fh dt + \dfrac{\partial V}{\partial t}(x,t) dt + \mathcal{O}\left( dt^2 \right), 
\end{align}
and 
\begin{align}
 \ML_\alpha (\a dt^\alpha) &= 1 + \dfrac{\a dt^{\alpha}}{\Gamma(\alpha + 1)} + \mathcal{O}(dt^{2\alpha}),
\end{align}
 we obtain the relation
\begin{align}
&{\ML_\alpha (\a dt^\alpha)} V(x + \fh dt,t + dt) \nt\\
&\quad \cong \left( 1 + \dfrac{\a dt^{\alpha}}{\Gamma(\alpha + 1)} \right) 
\left( V + \dfrac{\partial V}{\partial x}\fh dt + \dfrac{\partial V}{\partial t} dt \right). 
\end{align}
%
\end{proof}

\begin{lemma}
(\ref{eq4a}) can be written as
\begin{align}
 0 &= \inf_{u[t,t^\pr]} \left( L(x, u, t) dt + \dfrac{\partial V}{\partial x}\fh dt + \dfrac{\partial V}{\partial t} dt \right. \nt \\
&\quad + \left. \dfrac{\a dt^{\alpha}}{\Gamma(\alpha + 1)} V 
+ \dfrac{\a dt^{\alpha}}{\Gamma(\alpha + 1)} \dfrac{\partial V}{\partial x}\fh dt
+ \dfrac{\a dt^{\alpha}}{\Gamma(\alpha + 1)} \dfrac{\partial V}{\partial t} dt
\right) \nt \\
 &= \inf_{u[t,t^\pr]} \left( L(x, u, t) + \dfrac{\partial V}{\partial x}\fh(x, u, t) + \dfrac{\partial V}{\partial t}(x,t) \right. \nt \\
 &\qquad \qquad \left. + \lambda  A(\alpha) \dfrac{\pd^{1-\alpha}V(x,t)}{\pd t^{1-\alpha}} \right), \label{HJfd0}
\end{align}
where $t^\pr = t + dt$, and $A(\alpha) = {(1-\alpha)^{\alpha-1}}/{\alpha^{\alpha}}$. 
\end{lemma}

\begin{proof}
Let us consider the fractional derivative of composite functions~\citep[p. 98]{Pod:99}
\begin{align}
 \CDh{a}{t}{\alpha} f(g(t)) &= \dfrac{(t-a)^{-\alpha}}{\Gamma(1 - \alpha)}f(g(t)) \nt \\
&\quad + \sum_{k=1}^\infty C^\alpha_k \dfrac{k!\,(t-a)^{k-\alpha}}{\Gamma(k-\alpha+1)} 
 \sum_{m=1}^k \left( \CDg{g}{m} f(g) \right) \nt \\
&\qquad \cdot \sum \prod_{r=1}^k \dfrac{1}{a_r!} \left( \dfrac{\CDg{t}{r}g(t)}{r!} \right)^{a_r}, \label{eqGTFC}
\end{align}
where $t>0$, $\sum$ extends over all combinations of non-negative integer values of $a_1, a_2, \cdots, a_k$ 
such that $\sum_{r=1}^k r a_r = k$, and $\sum_{r}^k a_r = m$.
By substituting $t^\pr = t + dt$ to $t$ with $a = t$ in (\ref{eqGTFC}), we get
\begin{align}
 &\CDh{t}{t^\pr}{\alpha} V(x(t^\pr)) = \dfrac{dt^{-\alpha}}{\Gamma(1 - \alpha)}V(x(t^\pr)) \nt \\
&\quad + \sum_{k=1}^\infty C^\alpha_k \dfrac{k!\,dt^{k-\alpha}}{\Gamma(k-\alpha+1)} 
 \sum_{m=1}^k \left( \CDh{x}{x^\pr}{m} V(x^\pr) \right) \nt \\
&\qquad \cdot \sum \prod_{r=1}^k \dfrac{1}{a_r!} \left( \dfrac{\CDh{t}{t^\pr}{r}x(t^\pr)}{r!} \right)^{a_r}, \label{Comp1}
\end{align}
where $x^\pr = x(t^\pr)$, and
\begin{align}
 C^\alpha_k = \dfrac{\Gamma(\alpha+1)}{\Gamma(k+1)\Gamma(\alpha-k+1)}.
\end{align}
Consider the approximation up to $k = 1$ in (\ref{Comp1}) as follows:
\begin{align}
 &\CDh{t}{t^\pr}{\alpha} V(x(t^\pr)) = \dfrac{dt^{-\alpha}}{\Gamma(1 - \alpha)}V(x(t^\pr)) \nt \\
 &\ + C^\alpha_1 \dfrac{dt^{1-\alpha}}{\Gamma(2-\alpha)} 
 \left( \CDh{x}{x^\pr}{1} V(x^\pr) \CDh{t}{t^\pr}{1}x(t^\pr) \right) + \mcal{O}(dt^{2-\alpha}), \label{Comp2}
\end{align}
where $C^\alpha_1 = {\Gamma(\alpha+1)}/{\Gamma(\alpha)} = \alpha$, and 
$a_1 = 1$ is obtained from the relation $r \cdot a_r = k$ with $r = 1$ and $k = 1$.
Next, by substituting $1-\alpha$ to $\alpha$ in (\ref{Comp2}), 
and multiplying $dt$ from the right-side, the following relation is given
\begin{align}
 &\CDh{t}{t^\pr}{1-\alpha} V(x(t^\pr))\,dt 
= \left[\dfrac{dt^{\alpha-1}}{\Gamma(1 - (1-\alpha))}V(x(t^\pr)) \right.\nt \\
 &\quad + (1-\alpha) \dfrac{dt^{1-(1-\alpha)}}{\Gamma(2-(1-\alpha))} 
 \left( \CDh{x}{x^\pr}{1} V(x^\pr) \CDh{t}{t^\pr}{1}x(t^\pr) \right) \nt \\
 &\quad \left.+ \mcal{O}(dt^{2-(1-\alpha)})\right]dt \nt \\
&= \dfrac{dt^{\alpha}}{\Gamma(\alpha)}V(x(t^\pr)) 
 + (1-\alpha) \dfrac{dt^{\alpha}}{\Gamma(1+\alpha)} 
 \left( \dfrac{\pd V}{\pd x^\pr} \fh \right)\,dt \nt \\
 &\quad + \mcal{O}(dt^{2+\alpha}). \label{Comp3}
\end{align}
Furthermore, we augmented (\ref{Comp3}) as a multi-variable expression for $x = \left[x_1, x_2, \cdots, x_n, t \right]^\top$ as follows:
\begin{align}
 &\CDh{t}{t^\pr}{1-\alpha} V(x^\pr,t^\pr)\,dt \nt\\
 &= \left[\dfrac{dt^{\alpha-1}}{\Gamma(\alpha)}V(x^\pr,t^\pr) + (1-\alpha) \dfrac{dt^{\alpha}}{\Gamma(1+\alpha))} \right. \nt \\
 &\qquad \cdot \left( \CDh{x}{x_1^\pr}{1} V(x^\pr,t^\pr) \CDh{t}{t^\pr}{1}x_1(t^\pr) + \CDh{x}{x_2^\pr}{1} V(x^\pr,t^\pr) \CDh{t}{t^\pr}{1}x_2(t^\pr) \right. \nt \\ 
 &\qquad \left. \left. + \cdots + \CDh{t}{t^\pr}{1} V(x^\pr,t^\pr) \CDh{t}{t^\pr}{1}t^\pr \right)\right]dt + \mcal{O}(dt^{2+\alpha}) \nt \\
&= \alpha \dfrac{dt^{\alpha}}{\Gamma(\alpha+1)}V(x^\pr,t^\pr) 
+ (1-\alpha) \dfrac{dt^{\alpha}}{\Gamma(\alpha+1)} 
 \left( \dfrac{\pd V}{\pd x^\pr} \fh + \dfrac{\pd V}{\pd t^\pr}\right)dt \nt \\
 &\quad + \mcal{O}(dt^{2+\alpha}). \label{Comp4}
\end{align}
Finally, the following first order relation with respect to $dt$ has been obtained from (\ref{Comp4}):
\begin{align}
 &\dfrac{1}{\alpha} \CDh{t-dt}{t}{1-\alpha} V(x,t) dt \nt\\
&= \dfrac{dt^{\alpha}}{\Gamma(\alpha+1)}V(x,t) 
+ \dfrac{1-\alpha}{\alpha} \dfrac{dt^{\alpha}}{\Gamma(\alpha+1)} 
 \left( \dfrac{\pd V}{\pd x} \fh + \dfrac{\pd V}{\pd t}\right)dt.
\end{align}
Moreover, we get
\begin{align}
 &\dfrac{\Sp^{\alpha-1}}{\alpha} \CDh{t-dt}{t}{1-\alpha} V(x,t)\Sp \,dt \nt\\
&= \dfrac{(\Sp dt)^{\alpha}}{\Gamma(\alpha+1)}V(x,t) 
+ \Sp \dfrac{(\Sp dt)^{\alpha}}{\Gamma(\alpha+1)} 
 \left( \dfrac{\pd V}{\pd x} \fh + \dfrac{\pd V}{\pd t}\right)dt,
\end{align}
where we have defined $\Sp = (1-\alpha)/\alpha$.
We can substitute $\Sp dt$ with the scaling parameter $\Sp$ to $dt$ in the first equality of (\ref{HJfd0}) without loss of generality.
Hence, the second equality of (\ref{HJfd0}) is given. 
\end{proof}

\begin{theorem}
From (\ref{HJfd0}), the Hamilton-Jacobi-Bellman equation with extended discounted cost
\begin{align}
 &-{\a}A(\alpha) \dfrac{\pd^{1-\alpha}V(x,t)}{\pd t^{1-\alpha}} - \dfrac{\partial V}{\partial t}(x,t) \nt \\
 &\qquad= \inf_{u} \left( L(x, u, t) + \dfrac{\partial V}{\partial x}(x,t)\fh(x,u,t) \right) \nt \\
 &\qquad= \inf_{u} H(x, u, \p, t), \label{HJfd}
\end{align}
can be derived, where the pre-Hamiltonian
\begin{align}
 H(x, u, \p, t) = L(x, u, t)  + \dfrac{\partial V}{\partial x}(x,t)\fh(x,u,t) \label{preH}
\end{align}
with $\p = (\partial V/ \partial x)^\top(x,t)$ has been defined.
\end{theorem}

\memo{
\begin{figure}
\begin{center}
\includegraphics[width=0.75\columnwidth]{./funcA.eps}
\caption{\centering The shape of $A(\alpha)$}
\end{center}
\end{figure}
} 



\section{Conclusion}

This paper proposed the Hamilton-Jacobi-Bellman equation of nonlinear optimal control problems for cost functions including 
a fractional discount rate described by Mittag-Leffler functions.
The authors attempt to integrate this formulation into the stable manifold method~\citep{sakamoto2008} that is an exact numerical solver of
Hamilton-Jacobi equations in nonlinear optimal control problems.



\end{document}